\documentclass[10pt]{amsart}
\usepackage{amssymb, amscd, amsmath, amsthm, epsf, epsfig, latexsym, enumerate}

\newtheorem{theorem}{Theorem}
\newtheorem{lemma}[theorem]{Lemma}

\begin{document}

\title{On Farber's invariants for simple $2q$-knots}

\author{Jonathan A. Hillman }
\address{School of Mathematics and Statistics\\
     University of Sydney, NSW 2006\\
      Australia }

\email{jonathan.hillman@sydney.edu.au}

\begin{abstract}
Let $K$ be a simple $2q$-knot with exterior $X$. We show directly how the
Farber quintuple $(A,\Pi,\alpha,\ell,\psi)$ determines the homotopy type of $X$
if the torsion subgroup of $A=\pi_q(X)$ has odd order.
We comment briefly on the possible role of the EHP sequence in recovering 
the boundary inclusion from the duality pairings $\ell $ and $\psi$.
Finally we reformulate the Farber quintuple as an hermitian self-duality
of an object in an additive category with involution.
\end{abstract}



\maketitle

In a series of papers in the early 1980s Farber showed that
stable high dimensional knots could be classified in terms of
stable homotopy theory, and in particular that if $q\geq4$ 
simple $2q$-knots may classified up to isotopy by quintuples 
$(A,\Pi,\alpha, \ell,\psi)$, 
where $A,\Pi$ are $\mathbb{Z}[\mathbb{Z}]$-modules, 
$\alpha$ is a monomorphism and $\ell$ and $\psi$ are sesquilinear pairings
\cite{Fa81,Fa84}.
Certain significant special cases were dealt with earlier by
Kearton \cite{Ke83} and Kojima \cite{Ko79}.

Farber's argument for showing that his invariant determined the knot 
involves Spanier-Whitehead duality for highly connected Seifert hypersurfaces
and Wall's embedding theorem 
(to interpolate between different choices of such hypersurfaces).
An alternative approach (avoiding Seifert arguments)
is suggested by the work of Lashof and Shaneson,
who used surgery over $\mathbb{Z}$ to show that the exterior $X$ 
of a high dimensional knot with group $\mathbb{Z}$ is determined 
up to homeomorphism by the homotopy type of the pair 
$(X,\partial{X})$ \cite{LS69}.
There remains a possible ambiguity of order 2,
and a further argument is needed to show that the exterior determines the knot
if $(X,\partial{X})$ is highly connected \cite{Ri94}.

In this note we shall provide some evidence for the feasibility of
such a homotopy-theoretic approach.
We show first that the homotopy type of $X$ is determined by the modules 
$A=\pi_q(X)$ and $\pi_{q+1}(X)$ and a $k$-invariant $\kappa(X)$.
We then show that these invariants may be derived algebraically from 
the Farber quintuple.
(In the case of $\kappa(X)$ we need to assume that the $\mathbb{Z}$-torsion 
submodule $T(A)$ has odd order.)
The work of Farber implies that the duality pairings $\ell$ and $\psi$ 
determine the inclusion of the boundary.
We have not yet been able to find a homotopy-theoretic proof
of this implication.
However the longitude (i.e., the lift of the boundary inclusion to infinite
cyclic covers) has trivial suspension,
and the EHP sequence of Chapter XII of \cite{GWW} suggests a close connection
between the longitude and these duality pairings.
Further progress appears to depend on having a more direct construction
of the pairing $\psi$ (or at least of the version $\Psi$ used below
for the case with $T(A)$ odd).

Finally, Quebbemann, Scharlau and Schulte have codified the notion of hermitian 
pairing in an additive category $\mathcal{C}$ with an involution $*$ \cite{QSS},
and we shall show that Farber's invariant may be described in such terms.
This approach has proven useful in studying the factorization of some
nontrivial classes of higher-dimensional knots into sums of indecomposable knots 
\cite{Hi86, HK89}.
It is also of interest in connection with criteria for a knot to be doubly null
concordant \cite{FH93,Hi86}.

\section{Knots}

An {\it $n$-knot} is a locally flat embedding $K:S^n\to{S}^{n+2}$.
We shall assume that $S^n$ has the standard orientation as the boundary 
of the unit ball in $\mathbb{R}^{n+1}$, for all $n>0$.
The {\it exterior} of $K$ is $X=X(K)=S^{n+2}\setminus{N(K)}$, 
where $N(K)\cong{S^n}\times\mathbb{R}^2$ is an open regular neighbourhood 
of the image of $K$, and $\partial{X}\cong{S^n}\times{S^1}$.
The orientations of the spheres determine an orientation for $X$ and
an isotopy class of {\it meridians}  ${S^1}\subset\partial{X}$.
The inclusion of  a meridian induces an isomorphism on integral homology,
and so there is a canonical generator $t$ for 
$H=H_1(X(K);\mathbb{Z})\cong\mathbb{Z}$.
Hence we may identify $\mathbb{Z}[H]$ with 
$\mathbb{Z}[\mathbb{Z}]=\Lambda=\mathbb{Z}[t,t^{-1}]$.

Reattaching $S^n\times D^2$ to $X=X(K)$ along $\partial{X}\cong S^n\times{S^1}$
using the nontrivial twist map of $S^n\times{S^1}$ gives a homotopy
$(n+2)$-sphere. 
The embedding of $S^n\times\{0\}$ in this sphere gives an $n$-knot $K^*$,
called the {\it Gluck reconstruction} of $K$,
and any knot with exterior homeomorphic to $X$ is isotopic to $K$ or to $K^*$
(up to change of orientations).

Let $p:X'\to{X}$ be the infinite cyclic cover space corresponding to 
the commutator subgroup $\pi'$.
The modules $H_*(X;\Lambda)=H_*(X';\mathbb{Z})$ are finitely generated, 
since $X$ is compact and $\Lambda$ is noetherian.
Moreover $t-1$ acts invertibly on $H_*(X;\Lambda)$, 
since $X$ is an homology circle, and so these are torsion modules.
The {\it longitude} $\lambda_K\in\pi_n(X)$ is the class of the 
inclusion $S^n\times\{1\}\to{S^n}\times{S^1}\cong\partial{X}\subset{X}$.
If $V$ is a Seifert hypersurface for $K$ then $\lambda_K$ factors though
the inclusion of $\partial{V}\cong{S^n}$ into $V$.
Let $k:V/\partial{V}\to S^{n+1}$ be the cofibre of the natural map
from $V$ to $V/\partial{V}$.
Then the Pontrjagin-Thom collapse of
$S^{n+2}$ onto $S(V/\partial{V})$ provides a section to $Sk$.
It follows easily that the suspension of $\lambda_K$ is trivial \cite{Su}.

The knot $K$ is {\it $r$-simple} if $\pi_1(X)\cong\mathbb{Z}$ and
$\pi_j(X)=0$ for $1<j\leq{r}$.
The {\it sum} $K\#\widetilde{K}$ of $r$-simple $n$-knots $K$ and 
$\widetilde{K}$ is $r$-simple.
If $K$ is 1-simple then $X$ is determined up to homeomorphism by the homotopy 
type of the pair $(X,\partial{X})$ \cite{LS69}.
The knot $K$ is {\it simple} if it is $[(n-1)/2]$-simple and $n\geq3$.

It is a consequence of Farber's work on stable knots that an
$r$-simple $n$-knot $K$ is determined by its exterior
(i.e., $K\cong{K^*}$) if $3r>n$.
In \cite{Ri94} Richter shows that 
if $3r\geq{n}$ and $K$ is an $r$-simple $n$-knot the twist map 
of $\partial{X}$ extends to a self homotopy equivalence of $X$
and hence that $K$ is determined by its complement,
as a straightforward consequence of Williams' Poincar\'e Embedding Theorem.
Richter's approach is close to the one taken here,
in that he strives to work as far as possible in the homotopy category,
before appealing to geometric topology through surgery.
In particular, he subsequently gave a purely homotopy theoretic proof of
William's theorem \cite{Ri97}.
(The application to knot theory does however use the fact that $K$ has a 
highly connected Seifert hypersurface, in an essential way.)
Is there an {\it ad hoc\/} simplification for the case of simple knots? 

\section{Modules and extensions}

If $M$ is a finitely generated left $\Lambda$-module let $T(M)$ 
be the maximal finite submodule of $M$.
In the cases of interest here $T(M)$ is the $\mathbb{Z}$-torsion 
submodule of $M$.
The quotient $M/T(M)$ is $\mathbb{Z}$-torsion free, 
and has projective dimension at most 1 ($p.d._\Lambda{M/T(M)}\leq1$).
If $R$ is a {\it right} $\Lambda$-module let $\overline{R}$ denote the left
module with the conjugate structure determined by $t.r=r{t^{-1}}$ 
for all $r\in{R}$.

The extension modules ${Ext^i_\Lambda(M,\Lambda)}$ are naturally 
right modules.
Let $e^iM=\overline{Ext^i_\Lambda(M,\Lambda)}$.
If $M$ is a finite $\Lambda$-module then $e^0M=e^1M=0$,
and so the inclusion $T(M)\leq{M}$ induces isomorphisms
$e^iM\cong{e^i(M/T(M))}$ for $i=0,1$ and $e^2M\cong{e^2T(M)}$.
More generally, ${Ext}^2_\Lambda(M,N)\cong{Ext}^2_\Lambda(T(M),N)$,
for any $\Lambda$-module $N$.
If $p.d._\Lambda{M}=d$ then $Ext^d_\Lambda(M,N)$ and 
$Ext^d_\Lambda(M,\Lambda)\otimes_\Lambda{N}$ are naturally isomorphic,
for if $P_*$ is a finitely generated free resolution of $M$
then $Hom_\Lambda(P_*,N)\cong{Hom}_\Lambda(P_*,\Lambda)\otimes_\Lambda{N}$,
and $-\otimes_\Lambda{N}$ is right exact.

If $P_*$ is a chain complex of finitely generated free left $\Lambda$-modules
let $D(P)_j=\overline{Hom_\Lambda(P_{-j},\Lambda)}$ be the dual complex,
and let $P[n]_*$ be the $n$-fold suspension, with $P[n]_j=P_{j-n}$ for all $j$.
Note that $D(D(P))_*$ is naturally isomorphic to $P_*$.
If $H_*(P)=0$ for $j\not={q}$ or $q+1$ then $P$ is determined up to chain
homotopy equivalence by the modules $A=H_q(P_*)$ and $B=H_{q+1}(P_*)$
and a $k$-invariant $\kappa(P_*)$ in ${Ext}^2_\Lambda(A,B)$ \cite{Do60}.
This may be identified with the class of the sequence
$$0\to{B}\to{P_{q+1}}/\partial{P_{q+2}}\to\mathrm{Ker}(\partial_q)\to{A}\to0$$
in ${Ext}^2_\Lambda(A,B)\cong{Ext}^2_\Lambda(T(A),B)$.
Every such class is realizable by a free complex concentrated in degrees
$[q,q+2]$.
(Note also that ${Ext}^2_\Lambda(A,B)$ is finite,
since $T(A)$ and $B$ are finitely generated.)

A {\it simple $\Lambda$-complex concentrated in degrees $[q,q+2]$}
is a finitely generated free $\Lambda$-complex $C_*$ whose homology modules 
are $\Lambda$-torsion modules and are 0 except in degrees $q$ and $q+1$.
Let $A=H_q(C_*)$, $T=T(A)$ and $B=H_{q+1}(C_*)$ and let 
$k(A)\in{Ext}^1_\Lambda(A/T,T)\cong\overline{e^1A}\otimes{T}$ and 
$\kappa(C_*)\in{Ext}^2_\Lambda(A,B)\cong\overline{e^2T}\otimes{B}$ 
be the associated extension classes.

\begin{theorem}
Let $C_*$ be a simple $\Lambda$-chain complex concentrated in degrees 
${[q,q+2]}$.
\begin{enumerate}
\item  $C$ is determined up to chain homotopy equivalence by the quintuple
\[
(A/T(A),T(A),B,k(A),\kappa(C_*));
\]
\item  any such system of invariants with 
$A$ and $B$ finitely generated $\Lambda$-torsion modules and 
$p.d._\Lambda{B}\leq1$ is realized by such a chain complex;

\item $D_*=D(C)_*$ is a simple complex concentrated in degrees 
$[-q-2,-q]$, with invariants $(e^1B,e^2A,e^1A,k_D,\kappa(D_*))$
where $k_D$ and $\kappa(D_*)$ are determined by $\kappa(C_*)$ and $k(A)$,
respectively.
\end{enumerate}
\end{theorem}

\begin{proof}
The module $A$ is determined by $T=T(A)$, $A/T$ and $k(A)$,
and so the first assertion follows from the result of \cite{Do60} cited above.
Since $A$ is a torsion module $D_*$ is easily seen to be a simple
$\Lambda$-complex.
If $A_D=H_{-q-2}(D_*)$ and $B_D=H_{-q-1}(D_*)$ then $B_D\cong{e^1A}$, 
$T_D=T(A_D)\cong{e^2A}=e^2T$ and $A_D/T_D\cong{e^1B}$.

There are natural isomorphisms $T\cong{e^2e^2A}=e^2T_D$ and $B\cong{e^1A_D}$.
Hence the transpositions 
$\tau_1:\overline{e^1A}\otimes{T}\cong\overline{T}\otimes{e^1A}$ and
$\tau_2:\overline{e^2T}\otimes{B}\cong\overline{B}\otimes{e^2T}$ 
given by interchange of factors induce natural isomorphisms 
$\overline{e^1A}\otimes{T}\cong \overline{e^2T_D}\otimes{B_D}$
and $\overline{e^2T}\otimes{B}\cong \overline{e^1A_D}\otimes{T_D}$
between the abelian groups in which the extension classes for $A$ and $D_*$
and for $C_*$ and $A_D$ lie. 

To verify (2) and see that the extension classes correspond as asserted
we shall construct an explicit representative for
the chain homotopy type determined by the invariants 
$(A/T,T,B,k(A),\kappa(C_*))$.
We shall assume for simplicity of notation that $q=0$.
Let $P_1\to{P_0}$, $Q_1\to{Q_0}$ and $E_2\to{E_1}\to{E_0}$
be finite free resolutions of $A/T$, $B$ and $T$, respectively,
and let $k:P_1\to{E_0}$ and $\kappa:E_2\to{Q_0}$ be homomorphisms representing
$k(A)$ and $\kappa(C_*)$, respectively.
Then we may using a mapping cone construction (twice) as in Satz 7.6 of 
\cite{Do60} to construct a complex
\begin{equation}
\begin{CD}
{E_2}\oplus{Q_1}@>\Delta_2(\kappa)>>{P_1}\oplus{E_1}\oplus{Q_0}
@>\Delta_1(k)>>P_0\oplus{E_0}
\end{CD}
\end{equation}
with these invariants.
Here $\Delta_2(\kappa)$ and $\Delta_1(k)$ are matrices whose nonzero entries 
involve the differentials of the constituent complexes and the homomorphisms 
$k$ and $\kappa$.
It is clear that the dual complex is isomorphic to the complex obtained
by taking $e^0\kappa$ and $e^0k$ as representatives
for $k_D$ and $\kappa(D_*)$, respectively.
\end{proof}

In particular, if $D(C)_*\simeq{C_*}$ then the chain homotopy type of $C_*$ is determined 
by the module $A$ alone, for then $B\cong{e^1A}$ and $\kappa(C_*)$
is determined by $k(A)$.

\section{Duality for finite $\mathbb{F}\Lambda$-modules}

Let $\mathbb{F}$ be a field and let $\mathbb{F}\Lambda=\mathbb{F}[t,t^{-1}]$.
If $M$ is a finite dimensional $\mathbb{F}\Lambda$-module let 
$E(M)=\overline{Ext^1_{\mathbb{F}\Lambda}(M,\mathbb{F}\Lambda)}$,
$F(M)=\overline{Hom_{\mathbb{F}\Lambda}(M,\mathbb{F}(t)/\mathbb{F}\Lambda)}$
and $M^*=Hom_{\mathbb{F}}(M,\mathbb{F})$.
(The left $\Lambda$-module structure on the latter group is given by 
$(t\phi)(m)=\phi(t^{-1}m)$ for all $m\in{M}$ and $\phi:M\to\mathbb{F}$.)
These modules are each non-canonically isomorphic to $M$,
since $\mathbb{F}\Lambda$ is a PID.
The functors $E(-)$, $F(-)$ and $(-)^*$ each define an involution 
on the category of finite dimensional $\mathbb{F}\Lambda$-modules.
They are in fact naturally equivalent.
We shall give a simple proof of this for the case when $\mathbb{F}$ is finite.

If $f\in\mathbb{F}\Lambda$ then 
$\mathrm{Res}(f\frac{dt}t,\infty)+\mathrm{Res}(f\frac{dt}t,0)=0$,
since the only poles of $f\frac{dt}t$ are at 0 and $\infty$.
Thus we may define an $\mathbb{F}$-linear function
$R:\mathbb{F}(t)/\mathbb{F}\Lambda\to\mathbb{F}$ by setting
$R(q+\mathbb{F}\Lambda)=
\mathrm{Res}(q\frac{dt}t,\infty)+\mathrm{Res}(q\frac{dt}t,0)$
for all $q\in\mathbb{F}(t)$.
In particular,
if $q=\frac{f}{t^n-1}$ where $f\in\mathbb{F}[t]$ has degree $<n$ then
$R(q +\mathbb{F}\Lambda)=-f(0)$.

\begin{theorem}
There are natural equivalences $E(M)\cong{F(M)}\cong{M^*}$
on the category of finite dimensional $\mathbb{F}\Lambda$-modules, 
if $\mathbb{F}$ is a finite field.
\end{theorem}

\begin{proof}
Let $0\to{P_1}\to{P_0}\to{M}\to0$ be a free resolution of $M$.
If $\phi\in{F(M)}$ let $\phi_0:P_0\to \mathbb{F}(t)$ be a lift of $\phi$.
Then $\phi_0\partial_1$ has image in $\mathbb{F}\Lambda$, and so defines a 
homomorphism $\phi_1:P_1\to\mathbb{F}\Lambda$ such that $\phi_1\partial_2=0$.
Consideration of the short exact sequence of complexes
\begin{equation*}
0\to Hom_{\mathbb{F}\Lambda}(P_*,\mathbb{F}\Lambda)\to 
{Hom}_{\mathbb{F}\Lambda}(P_*,\mathbb{F}(t))\to
Hom_{\mathbb{F}\Lambda}(P_*,\mathbb{F}(t)/\mathbb{F}\Lambda)\to0
\end{equation*}
shows that $\delta_M(\phi)=[\phi_1]$,
where $\delta_M:F(M)\to{E(M)}$ is the Bockstein homomorphism associated to the
coefficient sequence.
(The extension corresponding to $\delta_M(\phi)$ 
is the pullback over $\phi$ of the sequence
$0\to\mathbb{F}\Lambda\to\mathbb{F}(t)\to
\mathbb{F}(t)/\mathbb{F}\Lambda\to0$.) 
Since $M$ is a torsion $\mathbb{F}\Lambda$-module and $\mathbb{F}(t)$ 
is divisible $\delta_M$ is an isomorphism.

Define $\chi_M:F(M)\to{M^*}$ by
$\chi_M(\phi)(m)=R(\phi(m))$ for all $\phi\in{F(M)}$ and $m\in{M}$.
Then $\chi_M$ is $\mathbb{F}\Lambda$-linear.
Suppose now that $M$ has finite order $n$.
If $g\in{M^*}$ let
$L_M(g)(m)=\frac{\hat{g}(m)}{t^{n!}-1}+\mathbb{F}\Lambda$, 
where $\hat{g}(m)=\Sigma_{0\leq{k}<n!}g(t^km)t^{-k}$ for all $m\in{M}$. 
Then $L_M(g)$ is $\mathbb{F}\Lambda$-linear, since ${(t^{n!}-1)M}=0$,
and $\chi_M(L_M(g))=g$.
(This is the only point at which we need $\mathbb{F}$ and $M$ to be finite).
Since $F(M)$ and $M^*$ have the same dimension 
and $\chi_M$ is $\mathbb{F}\Lambda$-linear
$\chi_M$ and $L_M$ are mutually inverse isomorphisms.
\end{proof}

Let $\tau_M=\chi_M\delta_M^{-1}:E(M)\to{M^*}$ be the composite isomorphism.

Theorem 2 holds also for infinite fields, but verifying that $\chi$ is
an equivalence involves slightly more work.
See \cite{Li,Tr78} for the case $\mathbb{F}=\mathbb{Q}$.

There is also a corresponding result over $\Lambda$, i.e.,
with integer coefficients. 
Levine showed that if $M$ is a finite $\Lambda$-module such that
$(t^k-1)M=0$ then 
$e^2M\cong\overline{Ext^1_\Lambda(M,\Lambda/(t^k-1))}\cong
\overline{Hom_\Lambda(M,\mathbb{Q}/\mathbb{Z}\otimes\Lambda/(t^k-1))}$
(via appropriate Bockstein homomorphisms),
which is in turn isomorphic to $Hom_\mathbb{Z}(M,\mathbb{Q}/\mathbb{Z})$
(via an explicit homomorphism similar to $L_M$) \cite{Le77}.
In particular, if $M$ has exponent $n$ as an abelian group then
$e^2M\cong{Hom_\mathbb{Z}(M,Z/nZ)}$.
If $A$ is a finitely generated $\Lambda$-module let $\sigma_A:e^2A\to
{Hom}_\mathbb{Z}(T(A),\mathbb{Q}/\mathbb{Z})$ be the composite of these
isomorphisms with the natural isomorphism $e^2A=e^2T(A)$.

\section{Algebraic invariants for the exterior of a simple $2q$-knot}

Let $K$ be a simple $2q$-knot with $q>2$.
Then $\pi_1(X)\cong\mathbb{Z}$ and $H_i(X;\Lambda)=0$ for $0<i<q$.
Let $A=H_q(X;\Lambda)$ and $B=H_{q+1}(X;\Lambda)$.
As observed in \S1, 
these are torsion $\Lambda$-modules on which $t-1$ acts invertibly.
The Universal Coefficient spectral sequence and Poincar\'e duality 
give isomorphisms $e^1A={e^1(A/T(A))}\cong{B}$ and
$e^2A\cong{T}(\overline{H^{q+2}(X;\Lambda)})\cong{T}(A)$,
while $H_i(X;\Lambda)=0$ for $i>q+1$.
In particular, there is a nonsingular pairing 
$\ell:T(A)\times{T}(A)\to\mathbb{Q}/\mathbb{Z}$ and $p.d._\Lambda{B}\leq1$.
Since $B=(t-1)B$ it follows that $B$ is $\mathbb{Z}$-torsion free \cite{Le77}.
The exterior $X$ fibres over $S^1$ if and only if $A$ is finitely generated as
an abelian group \cite{BL66}.

We may identify $A$ with the $\Lambda$-module $\pi_q(X)=\pi_q(X')$, 
by the Hurewicz Theorem.
Let $\widetilde{\pi}_{q+1}=\pi_{q+1}(X')=\pi_{q+1}(X)$, considered as a $\Lambda$-module,
and let $h:\widetilde\pi_{q+1}\to{B}$ be the Hurewicz homomorphism 
for $X'$ in degree $q+1$.
Let $\eta\in\pi_3(S^2)$ be the Hopf map.
Then $\eta_q=\Sigma^{q-2}\eta$ generates $\pi_{q+1}(S^q)=\pi_1^{st}=Z/2Z$, 
for all $q>2$.
Since $X'$ is $(q-1)$-connected and $q>2$ there is an exact sequence
\begin{equation}
\begin{CD}
\pi_{q+2}(X')\to{H_{q+2}}(X';\mathbb{Z})@>b>>
{A/2A}@>\eta_q^*>>\widetilde{\pi}_{q+1}@>h>>{H_{q+1}(X';\mathbb{Z})}\to0,
\end{CD}
\end{equation}
by Theorem XII.3.12 of \cite{GWW}.
Here $b$ is a ``secondary boundary homomorphism",
$\eta_q^*$ is induced by composition with $\eta_q$
and the unlabeled homomorphism is another Hurewicz homomorphism.
This sequence can also be derived from the Atiyah-Hirzebruch spectral sequence
for $\pi_*^{st}(X')$,
which gives another exact sequence
\begin{equation}
\begin{CD}
0\to{H}_q(X';\pi_2^{st})\to\Pi=\pi_{q+2}^{st}(X')@>\theta>>{H}_{q+1}(X';\pi_1^{st})\to0.
\end{CD}
\end{equation}
These are sequences of $\Lambda$-modules, by the naturality of
the spectral sequence.

In our situation $H_{q+2}(X';\mathbb{Z})=0$ and the right-hand portion of
sequence (2) reduces to a short exact sequence of $\Lambda$-modules
\[
0\to{A/2A}\to\widetilde\pi_{q+1}\to{B}\to0.
\]
The group $\widetilde{\pi}_{q+1}$ is stable: $\pi_{q+1}^{st}(X')\cong\pi_{q+1}(X')$,
since $X'$ is highly connected.
There is an analogous sequence when $q=2$, involving the universal quadratic
functor $\Gamma(A)$ instead of $A/2A$.
In this case $\widetilde{\pi}_{q+1}=\pi_3(X)$ is not a stable homotopy group,
there are nontrivial Whitehead products in the image of $\Gamma(A)$
and $\eta$ has infinite order.
(See \cite{HK98} for invariants for simple 4-knots.)

We may identify $H_q(X';\pi_1^{st})$ and $H_q(X';\pi_2^{st})$ with $A/2A$,
since $H_{q-1}(X';\mathbb{Z})=0$ and $\pi_1^{st}=\pi_2^{st}=Z/2Z$.
Composition (on the right) with $\eta_{q+1}$ induces a homomorphism 
$\eta_{q+1}^*:\widetilde\pi_{q+1}\to\Pi$, which factors through 
$\widetilde\pi_{q+1}/2\widetilde\pi_{q+1}$, since $\eta_{q+1}$ has order 2. 
Reduction {\it mod} (2)  induces a monomorphism $\rho_2$
from $B/2B$ with cokernel $Tor(A,Z/2Z)$, 
by the Universal Coefficient Theorem for homology.
Together, the sequences (2) and (3) give a commutative diagram 
\begin{equation}
\begin{CD}
0@>>>{A/2A}@>\eta_q^*>>\widetilde{\pi}_{q+1}/2\widetilde{\pi}_{q+1}@>>>{B/2B}@>>>0\\
@VVV  @VV=V     @VV\eta_{q+1}^*V        @VV\rho_2V                    @VVV \\
0@>>>{A/2A}@>\alpha>>\Pi@>\theta>>{H}_{q+1}(X';\mathbb{F}_2)@>>>0.
\end{CD}
\end{equation}

The chain complex of the covering space $X'$ is naturally a complex of 
$\Lambda$-modules, 
and is chain homotopy equivalent to a finitely generated free complex $C_*$.
As $H_j(C_*)=0$ if $j>q+2$ and $p.d._\Lambda{H_{q+1}}(C_*)\leq1$
we may assume that $C_*$ is the direct sum of the standard 
resolution of the augmentation $\Lambda$-module $\mathbb{Z}$ and a 
simple $\Lambda$-complex $C_*^{rel}$ concentrated in degrees $[q,q+2]$.
(This is homotopy equivalent to the equivariant chain complex of the pair
$(X,S^1)$ where $S^1$ is the meridian.)
Therefore $X$ is homotopy equivalent to a CW complex of dimension $q+2$ 
\cite{Wa65}.

Since $\pi_1(X)\cong\mathbb{Z}$ and $\pi_j(X)=0$ for $1<j<q$ 
the Postnikov $q$-stage of $X$ is the generalized Eilenberg-Mac Lane space
$L_{\mathbb{Z}}(A,q)$ determined by the $\Lambda$-module $A$.
(See page 214 of \cite{Ba89}.)
Let $\kappa(X)\in{H}^{q+2}(L_{\mathbb{Z}}(A,q);\widetilde{\pi}_{q+1})$ 
be the $k$-invariant for the next stage $P_{q+1}(X)$.
Since $H_{q+1}(K(A,q);\mathbb{Z})=0$ and $H_{q+2}(K(A,q);\mathbb{Z})\cong{A/2A}$,
by the Whitehead sequence for $K(A,q)$,
the Universal Coefficient spectral sequence for $L_{\mathbb{Z}}(A,q)$ 
with coefficients $\widetilde\pi_{q+1}$ gives an exact sequence
\begin{equation}
0\to{Ext}^2_\Lambda(A,\widetilde\pi_{q+1})\to
{H^{q+2}(L_{\mathbb{Z}}(A,q);\widetilde\pi_{q+1})}\to
{Hom}_\Lambda(A/2A,\widetilde\pi_{q+1})\to0.
\end{equation}
Since $H_{q+1}(K(A,q);\mathbb{Z})=0$,  
evaluation defines an isomorphism
\[
ev:H^{q+2}(K(A,q);\pi_{q+1}(X'))\to{Hom_\mathbb{Z}(A/2A,\pi_{q+1}(X'))}.
\]
Let $b_A$ be the secondary boundary homomorphism in degree $q+2$ for $K(A,q)$.

\begin{lemma}
The image $e$ of $\kappa(X)$ in $Hom_\Lambda(A/2A,\widetilde\pi_{q+1})$ 
is  $\eta_q^*b_A$.
It is a monomorphism and $p.d._\Lambda\mathrm{Cok}(e)\leq1$,
while $h_\#(\kappa(X))=\kappa(C_*)\in{Ext}^2_\Lambda(A,B)$.
\end{lemma}

\begin{proof}
Since ${Hom}_\Lambda(A/2A,\widetilde\pi_{q+1})\leq
{Hom}_\mathbb{Z}(A/2A,\pi_{q+1}(X'))$ and $\eta_q^*b_A$ is $\Lambda$-linear,
it shall largely suffice to work with $X'$.
We may reduce further to $P_{q+1}(X')$, since
the canonical map  $f:X'\to{P_{q+1}(X')}$ is $(q+2)$-connected, 
so $H_{q+1}(f)$ is an isomorphism and $H_{q+2}(P_{q+1}(X');\mathbb{Z})=0$,
and $f$ induces an isomorphism of the (truncated) Whitehead sequences 
for $X'$ and $P_{q+1}(X')$. 

The homomorphism $b_A$ is an isomorphism, 
by exactness of the Whitehead sequence for $K(A,q)$.
On comparing the definitions of $b_A$ and $\kappa(X')$,
as  in Theorems XII.3.12 and IX.2.2 of \cite{GWW}, respectively, 
we see that $ev(\kappa(X'))=\eta_q^*b_A$.
This is the image of $e$ under the inclusion
${Hom}_\Lambda(A/2A,\widetilde\pi_{q+1})\leq
{Hom}_\mathbb{Z}(A/2A,\pi_{q+1}(X'))$, and so $e$ is a monomorphism,
and $\mathrm{Cok}(e)\cong\mathrm{Cok}(\eta_q^*)=B$.
Hence $p.d._\Lambda\mathrm{Cok}(e)\leq1$.

The Universal Coefficient spectral sequence for $L_{\mathbb{Z}}(A,q)$ 
with coefficients $B$ gives an isomorphism 
${Ext}^2_\Lambda(A,B)\cong{H}^{q+2}(L_{\mathbb{Z}}(A,q);B)$,
since $B$ is $\mathbb{Z}$-torsion free.
Hence $h_\#(\kappa(X))=\kappa(C_*)$.
\end{proof}

The invariants discussed thus far are enough to characterize such knot exteriors. 

\begin{theorem}
Let $A$ and $\widetilde{\pi}_{q+1}$ be finitely generated $\Lambda$-torsion
modules and $\kappa\in{H}^{q+2}(L_{\mathbb{Z}}(A,q);\widetilde{\pi}_{q+1})$.
Suppose that the image of $\kappa$ in ${Hom}_\Lambda(A/2A,\widetilde\pi_{q+1})$
is a monomorphism with cokernel $B$,
and that $p.d._\Lambda{B}\leq1$.
Then there is a finite $(q+2)$-complex
$X_\kappa$ such that $\pi_1(X_\kappa)\cong\mathbb{Z}$,
$\pi_q(X_\kappa)\cong{A}$, $\pi_{q+1}(X_\kappa)\cong\widetilde{\pi}_{q+1}$,
$\widetilde{H}_j(X_\kappa;\Lambda)=0$ if $j\not=q$ or $q+1$,
and $\kappa(X)=\kappa$.
The homotopy type of $X_\kappa$ is determined by $A$, 
$\widetilde{\pi}_{q+1}$  and $\kappa$.
\end{theorem}

\begin{proof} 
These invariants determine a Postnikov $(q+1)$-stage $P=P_{q+1}(\kappa)$.
We may assume the $k$-skeleton $P^{[k]}$ is finite for all $k\geq0$, 
since $A$ and $\widetilde{\pi}_{q+1}$ are finitely generated $\Lambda$-modules 
and $\Lambda$ is noetherian.
The hypothesis on $\kappa$ implies that
$H_{q+1}(P^{[q+2]};\Lambda)=H_{q+1}(P;\Lambda)\cong{B}$ and that composition 
with $\eta_q$ determines a monomorphism from $A/2A$ to $\widetilde{\pi}_{q+1}$.
Hence the Hurewicz homomorphism for $P^{[q+2]}$ in degree $q+2$ is onto.
Since $p.d._\Lambda{B}\leq1$ and $P^{[q+2]}$ is a finite $(q+2)$-complex
$H_{q+2}(P^{[q+2]};\Lambda)$ is a finitely generated free module.
Attaching $(q+3)$-cells to $P^{[q+2]}$ along representatives for a basis 
of $H_{q+2}(P^{[q+2]};\Lambda)$ gives a finite $(q+3)$-complex 
with trivial homology in degrees $\geq{q+2}$,
and which is therefore homotopy equivalent to a finite $(q+2)$-complex
$X_\kappa$, by \cite{Wa65}.

Let $Y$ be a $(q+2)$-complex with Postnikov $(q+1)$-stage 
$f:Y\to{P}$ and such that $H_{q+2}(Y;\Lambda)=0$.
Then $f$ factors through a map $g:Y\to{X_\kappa}$,
since we may construct $P$ by adjoining cells of dimension $\geq{q+3}$ to
$X_\kappa$, to kill the higher homotopy groups.
The lift of $g$ to the universal covers induces isomorphisms on homology
and so is a homotopy equivalence,
by the Hurewicz and Whitehead Theorems.
Therefore $g$ is a homotopy equivalence.
\end{proof}

It is easy to construct a finite $(q+2)$-complex with 
equivariant chain complex of the required chain homotopy type, 
using only the surjectivity of the Hurewicz homomorphism in degree $q+1$
(and not appealing to \cite{Wa65}).

However not all such complexes admit Poincar\'e duality isomorphisms. 
Let $D_*=D(C^{rel})[-2q-2]_*$.
It is clear from Theorem 1 that there is a chain homotopy equivalence
$D_*\simeq{C_*^{rel}}$ if and only if $k(A)$ and $\kappa(C_*)$ correspond.

An analogous but simpler argument shows that if $K$ is a simple $(2q-1)$-knot
the homotopy type of $X=X(K)$ is determined by the module $A=H_q(X;\Lambda)$
alone.

In \S6 below we shall see that the invariants of Theorem 4 
are determined by the module $A$ and the homomorphism $\eta_q^*$,
when $T(A)$ is odd.
 
\section{Farber quintuples and Kearton $F$-forms}

In this section we shall define the Farber quintuple for simple $2q$-knots.
This reduces to the torsion linking pairing in the odd finite case, 
and to an $F$-form in the torsion-free case.
These cases were studied slightly earlier \cite{Ke83,Ko79}.
However, in the torsion-free case it is not clear that the $F$-form 
given by the Farber quintuple is the same as the one used in \cite{Ke83}.
In other, earlier work Kearton and Kojima applied different systems of invariants 
to the odd torsion case and to torsion-free, fibred simple knots, 
respectively \cite{Ke76,Ko77}.
Central to most of these approaches is the homotopy linking pairing
\[
\mathcal{L}_V:\pi_{q+1}(V)\times\pi_{q+1}(V)\to\pi_1^{st}=Z/2Z,
\]
for $V$ a $(q-1)$-connected Seifert hypersurface for $K$.
(The group $\pi_{q+1}(V)$ is stable,
and homotopy classes $u,v\in\pi_{q+1}(V)$ 
can be represented by embedded spheres,
since $V$ is highly connected.
These are unknotted in $S^{2q+2}$, since they have codimension $q+1>2$,
and so the pushoff of $v$ defines an element of 
$\pi_{q+1}(S^{2q+2}\setminus{u(S^{q+1})})=\pi_1^{st}$.)
The invariants of \cite{Ke76,Ko77} involve complicated equivalence relations,
and the connection with Farber's work is again not clear.

An {\it $F$-form} is a triple $(A,\mathcal{E},[-,-])$,
where $A$ is a finitely generated $\Lambda$-module with $T(A)$ odd 
and $t-1$ acting invertibly, 
$\mathcal{E}$ is an exact sequence of $\Lambda$-modules
\[
\mathcal{E}: 0\to{A/2A}\to\Pi\to{B/2B}\to0,
\]
and $[-,-]:\Pi\times\Pi\to\mathbb{F}_2$ is a nonsingular pairing on which $t$
acts isometrically, and such that the image of $A/2A$ is self-annihilating 
($[a,b]=0$ for all $a,b\in{A/2A}$) and the induced pairing  
$\langle-,-\rangle:A/2A\times{B/2B}\to\mathbb{F}_2$ is nonsingular.
In particular, $\Pi$ has exponent 2 and $|A/2A|=|B/2B|$.
(We have used the natural transformation $\chi_\Pi$ to modify Kearton's formulation;
he used a nonsingular hermitean pairing into $\mathbb{F}_2(t)/\mathbb{F}_2\Lambda$
instead.)

The $F$-form associated to a simple $2q$-knot $K$ has $A=H_q(X;\Lambda)$,
$\Pi=\widetilde\pi_{q+1}/2\widetilde\pi_{q+1}$ and $B=H_{q+1}(X;\Lambda)$,
and the sequence $\mathcal{E}$ is the {\it mod} (2) reduction of sequence (2) above.
The pairing $[-,-]$ defined by Kearton is defined for any simple $2q$-knot with $q\geq3$,
and is nonsingular if $T(A)$ is odd.
(The latter hypothesis is needed only for Lemmas 1.10 and 1.13 of \S1 of \cite{Ke83}.)
The pairing $\langle-,-\rangle$ is the Milnor duality pairing for $X'$ 
with coefficients $\mathbb{F}_2$, if $T(A)$ is odd.
The $F$-form is a complete invariant for torsion-free simple $2q$-knots 
with $q\geq4$  \cite{Ke83}; 
this was extended to torsion-free simple 6-knots in \cite{HK88}.

A {\it Farber quintuple} $(A,\Pi,\alpha, \ell,\psi)$ (in dimension $q$) consists of  a pair 
of finitely generated $\Lambda$-modules $A$ and $\Pi$ such that $t-1$ acts invertibly on $A$,
a monomorphism $\alpha:A/2A\to\Pi$ and nonsingular $(-1)^{q+1}$-symmetric
pairings 
\[
\ell:T(A)\times{T(A)}\to\mathbb{Q}/\mathbb{Z}\quad\mathrm{and}\quad 
\psi:{\Pi\times\Pi}\to{Z/4Z}
\] 
on which $t$ acts isometrically.
In other words the {\it adjoint} functions $Ad(\ell)$ and $Ad(\psi)$ are 
$\Lambda$-isomorphisms,
where $Ad(\ell):T(A)\to{Hom}_\mathbb{Z}(T(A),\mathbb{Q}/\mathbb{Z})$ 
is given by $Ad(\ell)(b)(a)=\ell(a,b)$ for all $a,b\in{T(A)}$,
and $Ad(\psi)$ is defined similarly.

Let $\beta=\alpha^*Ad(\psi):\Pi\to{Hom}_{\mathbb{F}_2}(A/2A,\mathbb{F}_2)$,
so that $\beta(p)(a)=\psi(p,\alpha(a))$ for all $a\in{A/2A}$ and $p\in\Pi$.
Then $\beta$ is an epimorphism and $\mathrm{Ker}(\beta)=\mathrm{Im}(\alpha)$.
These pairings interact in the following way.
If $M$ is a $\Lambda$-module and $m\in{M}$
let $[m]_2$ be the image of $m$ in $M/2M$.
Let $\gamma:\Pi\to{T(A)}$ be the homomorphism determined by 
$\psi(p,\alpha([x]_2))=\ell(\gamma(p),x)$ for all $x\in{T(A)}$ and $p\in\Pi$
and nonsingularity of $\ell$.
Then $\alpha([\gamma(p)]_2)=2p$, for all $p\in\Pi$.
In particular, Farber's $\Pi$ has exponent 4,
and $\gamma=0$ if and only if $T(A)$ has odd order.

The  Farber quintuple of $K$ is $q_F(K)=(A,\Pi,\alpha, \ell,\psi)$, where
$A=H_q(X;\Lambda)$, $\Pi=\pi_{q+2}^{st}(X')$,
$\alpha:A/2A\to\Pi$ is the monomorphism determined by composition with 
$\eta_q\eta_{q+1}$,
$\ell:T(A)\times{T(A)}\to\mathbb{Q}/\mathbb{Z}$ is the torsion linking pairing 
and $\psi:{\Pi\times\Pi}\to{Z/4Z}<\pi_3^{st}$ is a bilinear pairing derived from 
Spanier-Whitehead duality for a $(q-1)$-connected Seifert hypersurface $V$ 
for $K$ \cite{Fa84}.

In outline, Farber's construction of $\psi$ is as follows.
Let $N(V)$ be a regular neighbourhood of $V$ in $S^{2q+2}$ 
and let $Y=S^{2q+2}\setminus{N(V)}$.
We may assume that $V\cup{Y}\subset\mathbb{R}^{2q+2}=S^{2q+2}\setminus\{\infty\}$.
Then $(v,y)\mapsto\frac{v-y}{||v-y||}$ defines a map from $V\times{Y}$ 
to $S^{2q+1}$ which is nullhomotopic on $V\vee{Y}$ and hence
induces a (Spanier-Whitehead) map $SW_V:V\wedge{Y}\to{S}^{2q+1}$.
The orientations for $K$ and $S^{2q+2}$ determine an orientation for $V$ and
an oriented normal vectorfield,
and hence a preferred pushoff map $i_+:V\to{Y}$.
Let $u:V\wedge{V}\to{S^{n+1}}$ and $z:V\to{V}$ be stable maps
corresponding to the duality pairing and carving map of \cite{Fa84}, 
respectively.
Let 
\[\psi_V(x,y)={u\circ(x\wedge{y})}\in\pi_{2q+4}^{st}(S^{2q+1})=
\pi_3^{st}\cong{Z/24Z}
\]
for all $x,y\in\pi_{q+2}^{st}(V)$.
Let $P=\mathbb{Z}[z]$, $\bar{z}=1-z$ and $L=P[(z\bar{z})^{-1}]$.
If we identify $t\in\Lambda$ with $1-z^{-1}\in{L}$
then we also have $L=\Lambda[(1-t)^{-1}]$.
Then $\pi_{q+2}^{st}(X')\cong{L}\otimes_P\pi_{q+2}^{st}(V)$ and $\psi_V$ 
extends to a pairing $\psi$ on $\pi_{q+2}^{st}(X')$ which is independent of 
the choice of Seifert hypersurface $V$.
(See \S10 of \cite{Fa84}.)

Suppose now that $T(A)$ has odd order.
Then $\Pi$ has exponent $2$, $\ell$ and $\psi$ are independent and $\psi$ is symmetric. 
In this case, a Farber quintuple is algebraically equivalent to an $F$-form together
with a torsion linking pairing.
If $A=2A$ then $\Pi=0$ and $q_F(K)$ reduces to the pair $(A,\ell)$.
In particular, if $A$ is finite this gives Kojima's classification of
odd finite simple $2q$-knots  \cite{Ko79}.
If $A$ is torsion-free $q_F(K)$ reduces to an $F$-form.
We can show that the exact sequence $\mathcal{E}$ of this reduction agrees
with that of Kearton, but it is not clear how the pairings $\psi$ and $[-,-]$ correspond.

Farber's construction gives an analogous pairing $\Psi$ on 
$\pi_{q+1}^{st}(X')\cong{L}\otimes_P\pi_{q+1}^{st}(V)$.
Let 
\[\Psi_V(a,b)={u\circ(a\wedge{b})}\in
\pi_{2q+2}^{st}(S^{2q+1})=\pi_1^{st}={Z/2Z}
\] 
for all $a,b\in\pi_{q+1}^{st}(V)$.
Then $\Psi_V$ extends to a pairing $\Psi$ on $\pi_{q+1}^{st}(X')$, 
which is again independent of the choice of Seifert hypersurface $V$.

\begin{lemma}
If $T(A)$ has odd order then 
$\widetilde\pi_{q+1}/2\widetilde\pi_{q+1}\cong\Pi=\pi_{q+2}^{st}(X')$,
and $\psi$ and $\Psi$ are equivalent.
\end{lemma}

\begin{proof}
Since $T(A)$ has odd order, $Tor(A,Z/2Z)=0$,
and so the homomorphism  $\rho_2$ in diagram (4) is an isomorphism.
Hence $\eta^*_{q+1}:\widetilde\pi_{q+1}/2\widetilde\pi_{q+1}\cong\Pi$
is also an isomorphism.
Composition on the right with $\eta^2$ induces a monomorphism from $\pi_1^{st}$ 
to $\pi_3^{st}$,
and  $\eta\wedge\eta=\eta^2\in\pi_2^{st}$, by X.8.12 of \cite{GWW}.
Hence $\Psi(a,b)=\psi(a\eta_{q+1},b\eta_{q+1})$ for all 
$a,b\in\pi_{q+1}^{st}(X')$, and so $\eta^*_{q+1}$ is an isometry.
\end{proof}

We expect that  $\Psi=[-,-]$.
If so, then the classifications of torsion-free simple $2q$-knots by Farber and Kearton 
are equivalent.

The $F$-form and the torsion linking pairing have the advantage 
of being defined in terms of invariants of the exterior.
Are there direct definitions of $\psi$ or $\Psi$
which does not involve the choice of a Seifert hypersurface?
For instance, can they be defined directly in terms of 
a $\mathbb{Z}$-equivariant SW duality?
(See also \cite{Fa81} for the case of fibred simple $2q$-knots.)

These invariants are defined for $q\geq1$, 
but simple $2$-knots are topologically trivial, 
and the natural invariants for simple $4$-knots are not stable,
so it is unlikely that the Farber quintuple 
is a complete invariant when $q=2$.
(See instead \cite{HK98}.)
It may however suffice when $q=3$.
There is an obvious notion of sum of Farber quintuples,
and $q_F(K\#\widetilde{K})=q_F(K)\oplus{q_F(\widetilde{K})}$.
Can the Farber invariant be extended to all even-dimensional knots, 
as an additive invariant?

\section{Recovering the homotopy type of $X$ from $q_F(K)$}

The first three ingredients of $q_F(K)$ are determined by the homotopy type of
$X$ alone, while $\ell$ and $\psi$ are manifestations of Poincar\'e duality
for the pair $(X,\partial{X})$.
We shall show that, conversely,
$q_F(K)$ determines $\widetilde\pi_{q+1}$ and $\kappa(X)$ and hence the
homotopy type of $X$, at least if $T(A)$ has odd order.
We shall consider how the duality pairings interact with the
inclusion of $\partial{X}$ into $X$ in the next section.

Let $pd_Z:B=H_{q+1}(X;\Lambda)\to{e^1A}$
and $pd_2:H_{q+1}(X;\mathbb{F}_2\Lambda)\to{E(A/2A)}=
\overline{Ext_{\mathbb{F}_2\Lambda}^1(A/2A,\mathbb{F}_2\Lambda)}$
be the isomorphisms determined by Poincar\'e duality and the Universal
Coefficient spectral sequence,
and let $\tau:E(A/2A)\cong{Hom}_{\mathbb{F}_2}(A/2A,\mathbb{F}_2)$ 
be the isomorphism defined in \S2.
Reduction {\it mod} (2) defines a homomorphism
$\rho:B\to{H_{q+1}(X;\mathbb{F}_2\Lambda)}$.
It  also induces a natural isomorphism $\mathbb{F}_2\otimes_\mathbb{Z}e^1(A/T(A))\cong
{E}(A/(T(A),2A))$, and hence a natural transformation
\[
e:e^1A\cong{e^1}(A/T(A))\to{E}(A/(T(A),2A))\to{E}(A/2A),
\]
and the diagram
\begin{equation}
\begin{CD}
B@>>>\overline{H^{q+1}(X;\Lambda)}@>>>{e^1A}\\
@VV{\rho}V  @VVV     @VVeV \\
H_{q+1}(X;\mathbb{F}_2\Lambda)@>>>\overline{H^{q+1}(X;\mathbb{F}_2\Lambda)}
@>>>E(A/2A)
\end{CD}
\end{equation}
is commutative.
(The horizontal maps are all isomorphisms.)

Let $\beta=\alpha^*Ad(\psi):\Pi\to{Hom}_{\mathbb{F}_2}(A/2A,\mathbb{F}_2)$,
as in \S5 above.

\begin{lemma}
The module $\widetilde\pi_{q+1}$ is
the pullback of $\tau{e}$ and $\beta$ over 
$Hom_{\mathbb{F}_2}(A/2A,\mathbb{F}_2)$, i.e., it is the fibre sum 
$\widetilde\pi_{q+1}\cong\{(b,p)\in{e^1A}\times\Pi\mid \tau{e}(b)=\beta(p)\}$.
\end{lemma}

\begin{proof}
This follows easily from diagrams (4) and (6),
since $\beta=\tau{pd_2}\theta$
(see Theorem 3.3(b) of \cite{Fa84}).
\end{proof}

\begin{lemma}
The equivariant chain complex of the universal cover is determined 
up to chain homotopy equivalence by the module $A$.
\end{lemma}

\begin{proof}
As observed above, 
the equivariant chain complex of the universal cover 
is chain homotopy equivalent to a complex $C_*$ which is
the direct sum of the standard resolution of the augmentation $\Lambda$-module 
$\mathbb{Z}$ and a simple $\Lambda$-complex $C_*^{rel}$ concentrated in degrees $[q,q+2]$.
Poincar\'e duality for $(X,\partial{X})$ reduces to a chain homotopy equivalence
$D(C^{rel})[-2q-2]_*\simeq{C_*^{rel}}$, and so it follows from Theorem 1 that
$C_*^{rel}$ and $C_*$ are determined up to chain homotopy equivalence by $A$. 
\end{proof}

\begin{theorem}
If $T(A)$ has odd order the homotopy type of $X$ is determined by $q_F(K)$.
\end{theorem}

\begin{proof}
The quintuple $q_F(K)$ determines $\widetilde\pi_{q+1}$, by Lemma 5.
If $T(A)$ has odd order then $Ext^2_\Lambda(A,A/2A)=Ext^2_\Lambda(T(A),A/2A)=0$.
Hence 
\[Ext^2_\Lambda(A,\widetilde\pi_{q+1})\cong
{Ext}^2_\Lambda(A,B)\cong{Ext}^2_\Lambda(T(A),e^1(A)),
\]
and is the odd-order summand of
$H^{q+2}(L_{\mathbb{Z}}(A,q);\widetilde\pi_{q+1})$.
The image of $\kappa(X)$ in the 2-primary summand is $\eta_q^*b_A$,
by Lemma 3. This corresponds to $\alpha$.
Hence  it follows that $\kappa(X)$ is determined by $\kappa(C_*^{rel})$
and $\alpha$, and so in turn by $A$ and $\alpha$, by Lemma 7.
Thus $q_F(K)$ determines the homotopy type of $X$, by Theorem 4.
\end{proof}

If $T(A)$ is odd $\alpha$ is the composite of the monomorphism $\eta^*_q$ 
with reduction {\it mod} $(2)$ and
$q_F(K)$ is algebraically equivalent to the torsion linking 
pairing $\ell$ together with an $F$-form.
If $A=2A$ then $\Pi=0$, $\widetilde\pi_{q+1}\cong{e^1A}$ and 
$\kappa(X)=\kappa(C_*^{rel})\in{Ext}^2_\Lambda(A,e^1A)$.
If instead $A$ is torsion free ${Ext}^2_\Lambda(A,\widetilde\pi_{q+1})=0$ 
and so $\kappa(C_*^{rel})=0$.
In this case $\kappa(X)$ is determined by $\alpha$,
which lifts to a monomorphism 
$A/2A\cong\{(0,p)\in{B}\times\Pi\mid \beta(p)=0\}\leq\widetilde\pi_{q+1}$
in the above fibre sum.

\section{Inclusion of the boundary}

Let $Y=S^{2q}\times{S^1}$ have the product orientation
and let $J_K:Y\cong\partial{X}\subset{X}$ be the inclusion of the boundary.
The oriented homotopy type of $X$ and the class $[J(K)]$ of $J_K$ modulo 
homotopy and composition with orientation preserving self homotopy equivalences
of $X$ and $Y$ together determine the oriented homotopy type of the pair 
$(X,\partial{X})$.
The difficulty is in giving a practical characterization of $[J(K)]$.
In this section we shall mention two of the known constraints.

Let $J:Y\to{X}$ be a map which induces an isomorphism on fundamental groups.
Let $M(J)$ be the mapping cylinder of $J$.
There is a canonical homotopy equivalence $M(J)\simeq{X}$. 
The orientation for $Y$ determines a generator $\mu_J$ for 
${H}_{2q+2}(M(J),Y;\mathbb{Z})\cong{H}_{2q+1}(Y;\mathbb{Z})$.
Since the universal covers of $M(J)\simeq{X}$ and $Y$ are highly connected
cap product with $\mu_J$ determines homomorphisms 
$\overline{H^{q+1}(X;\Lambda)}\to{H}_{q+1}(X;\Lambda)$
and $\overline{H^{q+2}(X;\Lambda)}\to{H}_q(X;\Lambda)$.
The pair $(M(J),Y)$ is a $PD_{2q+2}$-pair if and only if these
homomorphisms are isomorphisms.
(The corresponding homomorphisms in other degrees are isomorphisms, 
in all cases.)

The map $J':Y'=S^{2q}\times\mathbb{R}\to{X'}$ covering $J$ 
determines an element $\lambda_J\in\pi_{2q}(X')=\pi_{2q}(X)$, 
such that $(t-1)\lambda_J=0$.
In the fibred case the Poincar\'e duality constraints depend only 
on $\lambda_J$.
The longitude $\lambda_K$ corresponds to $J_K$.
Lemma 3 of \cite{HK98} asserted that 
{\sl if $X$ is an $n$-dimensional homology circle with $H_n(X;\Lambda)=0$ 
and two maps $j_1,j_2:S^n\times{S^1}\to{X}$ induce isomorphisms on $\pi_1$ 
and agree on $S^n\vee{S^1}$ then there is a self-homotopy equivalent 
$f:X\to{X}$ such that $j_2\sim{fj_1}$.}
If this were correct it would follow not only that $\lambda_K$ determines $J_K$ 
for 1-simple knots but also that such knots are determined by their exteriors.
The proof of Lemma 3 in \cite{HK98} is wrong, and thus the later
assumptions in that paper that the homotopy quadruple discussed 
there is a {\it complete\/} invariant for simple 4-knots are unjustified.
(We thank W.Richter for pointing out our error).
Nevertheless it remains possible that the result may be true, at least if
$X'$ is sufficiently highly connected.

The inclusion of the boundary of a Seifert hypersurface $V$ for
$K$ has trivial suspension in $\pi_{2q+1}(SV)$ \cite{Su}.
Therefore $\lambda_K$ is in the kernel
of the suspension $E:\pi_{2q}(X')\to\pi_{2q+1}(SX')$.
The group $\pi_{2q}(X')$ is just outside the stable range,
since $X'$ is $(q-1)$-connected, 
and sits near one extremity of the EHP sequence (Theorem XII.2.2 of \cite{GWW}):
\[\pi_{3q-2}(X')\to\dots\to\pi_{2q+2}(SX')\to\pi_{2q+2}(X'*X')\to
\pi_{2q}(X')\to\pi_{2q+1}(SX')\to0.\]
Moreover $X'\simeq{M_A}\vee{M_B}$ is a wedge of Moore spaces
$M_A=M(A,q)$ and $M_B=M(B,q+1)$,
since $X'$ is $(q-1)$-connected, $H_j(X';\mathbb{Z})=0$ for $j>q+1$ and $q>2$.
Hence $\pi_{2q}(X')\cong\pi_{2q}(M_A)\oplus\pi_{2q}(M_B)\oplus[A,B]$,
where $[A,B]$ is the subgroup generated by all Whitehead products $[x,y]$
with $x\in\pi_q(M_A)\cong{A}$ and $y\in\pi_{q+1}(M_B)\cong{B}$.
The summand $\pi_{2q}(M_B)$ is stable, 
while all Whitehead products have trivial suspension.

Suppose now that $K$ is fibred and $A\cong\mathbb{Z}^\beta$ is torsion-free.
Then $M_A\simeq\vee^\beta{S^q}$ and $M_B\simeq\vee^\beta{S^{q+1}}$.
Let $\{i^m_q\}_{m\leq\beta}$ and $\{i^n_{q+1}\}_{n\leq\beta}$
be bases for $A$ and the image of $B\cong\pi_{q+1}(\vee^\beta{S^{q+1}})$ 
in $\pi_{q+1}(X')$, respectively.
In this case $\mathrm{Ker}(E)$ is generated by all Whitehead products 
$i_{mm'}=[i^m_q,i^{m'}_q\eta_{q+1}]$ 
and $j_{mn}=[i^m_q,i^{n}_{q+1}]$ for all $m,m',n\leq\beta$.
Thus $\lambda_K=\Sigma c_{mm'}i_{mm'}+\Sigma d_{mn}j_{mn}$
for some coefficients $c_{mm'}\in\mathbb{F}_2$ and $d_{mn}\in\mathbb{Z}$.
The integral coefficients $d_{mn}$ may be detected by 
computing cup products in $(S^q\vee{S^{q+1}})\cup_{p\lambda}e^{2q+1}$
for suitable projections $p:X'\to{S^q\vee{S^{q+1}}}$.
(See \cite{Ja57}, and note that 
$S^q\vee{S^{q+1}}\cup_{[i,j]}e^{2q+1}\simeq{S^q}\times{S^{q+1}}$.)
Hence the matrix $D=[d_{mn}]$ is invertible,
since $(X',\lambda_K)$ is a $PD_{2q+1}$-complex,
and we may choose the basis for $B$ so that $D=I_\beta$.

\smallskip
\noindent{\bf Question.}
{\sl
How are the coefficients $c_{mm'}$ determined by the pairing $\Psi$?
}

\smallskip
An analogous but simpler argument applies if $K$ is a
fibred simple $(4k-1)$-knot.
For then $X'\simeq\vee^\beta{S^{2k}}$ and $\mathrm{Ker}(E)\leq\pi_{4k-1}(X')$ 
is generated by the Whitehead products $[i^m_{2k},i^n_{2k}]$ for $m,{n}\leq\beta$.
The coefficients $e_{mn}$ of $\lambda_K=\Sigma e_{mn}[i^m_{2k},i^n_{2k}]$
may be detected by cup products in $S^{2k}\cup_{p\lambda}e^{4k}$
or $(S^{2k}\vee{S^{2k}})\cup_{p\lambda}e^{4k}$
for suitable projections $p:X'\to S^{2k}$ 
(and using the Hopf invariant, as on page 211 of \cite{Ma})
or $p:X'\to S^{2k}\vee{S^{2k}}$ (and using \cite{Ja57}).

When $T(A)\not=0$ there may be elements of $\mathrm{Ker}(E)$ which are not 
Whitehead products, as can be seen already when $A=Z/pZ$.
This is the situation that arises for odd finite simple $2q$-knots,
for which $A=T(A)$ is finite of odd order and the Farber
quintuple reduces to the torsion linking pairing $\ell$.
(See \cite{Ko79}.)

\section{Hermitian self-dualities}

Let $\mathcal{C}$ be an additive category with an involution $*$ and let
$\varepsilon=\pm1$.
An $\varepsilon$-hermitian pairing on an object $C$ of $\mathcal{C}$ is
defined to be an isomorphism $\phi:C\to{C}^*$ such that 
$\phi^*=\varepsilon\phi$.
This categorical point of view is particularly useful if the 
Krull-Schmidt Theorem holds for the objects of $\mathcal{C}$
and the endomorphism rings of objects are radically complete.
For then the analysis of the decompositions of hermitian pairings 
into orthogonal direct sums may be reduced to corresponding questions for
pairings over (skew) fields \cite{QSS}.

Consider triples $Q=(A,\theta,\mathcal{E})$ where $A$ is a 
finitely generated $\Lambda$-module on which $t-1$ acts invertibly,
$\theta:T(A)\to{e^2A}$ is an isomorphism
and $\mathcal{E}$ is an exact sequence of $\Lambda$-modules
\begin{equation}
\begin{CD}
0\to{A/2A}@>\alpha>>\Pi@>\omega>>{e^2(A/2A)}\to0.
\end{CD}
\end{equation}
A morphism between two such triples $Q$ and $Q'$ is a triple 
$\phi=(f,g,h)$, where $f,g:A\to{A'}$ and $h:\Pi\to\Pi'$, 
such that $\theta'f|_{T(A)}=e^2g\theta$,
$h\alpha=\alpha'[f]_2$ and $\omega'h=e^2[g]_2\omega$.
The category $\mathcal{Q}$ of such triples is additive.

Define a duality functor by 
$Q^*=(A,\theta,\mathcal{E})^*=(A,e^2\theta,e^2\mathcal{E})$
and $(f,g,h)^*=(g,f,e^2h)$.
An isomorphism $\phi=(f,g,h)$ from $Q$ to $Q^*$ is
an {\it $\varepsilon$-hermitian self-duality} of $Q$
if $\phi^*=\varepsilon\phi$,
in which case $f$ and $g$ are automorphisms of $A$,
$f=\varepsilon{g}$ and $h=\varepsilon{e^2h}$.
A pair $(Q,\phi)$ with $\phi$ an $\varepsilon$-hermitian self-duality of $q$ 
is {\it hyperbolic} if there is an object $N$ in $\mathcal{Q}$ such that 
$Q\cong{N}\oplus{N^*}$ and $\phi$ has matrix 
$\left(\begin{smallmatrix}
0&\varepsilon\\ 1&0
\end{smallmatrix}\right)$ 
with respect to this decomposition.

\begin{theorem}
Farber quintuples for simple $2q$-knots correspond bijectively to
pairs $(Q,\phi)$ where $Q=(A,\theta,\mathcal{E})$ is an object in 
$\mathcal{Q}$ and $\phi=(f,g,h)$ is a $(-1)^{q+1}$-hermitian self-duality 
of $Q$ such that $\alpha[\gamma(p)]_2=2p$ for all $p\in\Pi$,
where $\gamma(p)$ is determined by
$\sigma_A(e^2g\theta(\gamma(p)))(x)=\sigma_\Pi{h}(\alpha([x]_2))(p)$, 
for all $x\in{T(A)}$ and $p\in\Pi$.
\end{theorem}

\begin{proof}
Let $(A,\Pi,\alpha,\ell,\psi)$ be a Farber quintuple
with $(-1)^{q+1}$-symmetric pairings $\ell$ and $\psi$  
and let $\beta$ be the associated epimorphism.
Let $\theta= \sigma_A^{-1}Ad(\ell)$, $\omega=\sigma_A^{-1}\beta$ and
$\mathcal{E}$ be the exact sequence determined by $\alpha$ and $\omega$.
Then $Q=(A,\theta,\mathcal{E})$ is in $\mathcal{Q}$ and 
$(1_A,1_A,\sigma_\Pi^{-1}Ad(\psi))$ is a
$(-1)^{q+1}$-hermitian self-duality of $Q$.

Conversely, if $\phi=(f,g,h)$ is a $(-1)^{q+1}$-hermitian self-duality of
$(A,\theta,\mathcal{E})$ we may define pairings $\ell$ on $T(A)$
and $\psi$ on $\Pi$ by setting $\ell(a,a')=e^2g\theta(a')(a)$ for all 
$a,a'\in{T(A)}$ and $\psi(p,q)=\sigma_\Pi(h(q))(p)$ for all $p,q\in\Pi$.
(We then have $\beta=\sigma_Ae^2g\omega$.)
Then $(A,\Pi,\alpha,\ell,\psi)$  is almost a Farber quintuple; 
we need only the additional condition that $\alpha[\gamma(p)]_2=2p$ 
for all $p\in\Pi$, where $\gamma$ is defined as in \S5.
The equation in the statement of the theorem is a reformulation of this
condition in terms of the constituents of $\phi$.
\end{proof}

The Krull-Schmidt Theorem holds in any additive category in which 
\begin{enumerate}
\item  all idempotents split; and

\item every object is a finite sum of objects whose endomorphism rings are local
rings.
(Such summands are necessarily indecomposable).
\end{enumerate}
It is easily verified that all idempotents in $\mathcal{Q}$ split.
We shall use the next lemma to determine the objects of $\mathcal{Q}$ 
which satisfy the second condition. 

\begin{lemma}
Let $R$ be a local $\Lambda$-algebra which is finitely generated as a
$\Lambda$-module. Then $R$ is finite.
\end{lemma}

\begin{proof}
Let $p:\Lambda\to{R}$ be the natural homomorphism.
Let $J=\mathrm{rad}(R)=R\setminus{R^\times}$ and $m=p^{-1}(J)$.
The skewfield $R/J$ is finitely generated as a $\Lambda/m$-module.
It follows easily that $\Lambda/m$ is a field.
In particular $m$ is a maximal ideal (and $R/J$ is a finite field).
Since $p(\Lambda)_m\leq{R}$ it is an integral extension of $p(\Lambda)$,
and every prime ideal of $p(\Lambda)$ is the restriction 
of a prime ideal of $p(\Lambda)_m$, by Theorem 5.10 of \cite{AM}.
Thus $p(\Lambda)$ is also a local ring, and so $m$ is the unique maximal ideal
of $\Lambda$ which contains $K=\mathrm{Ker}(p)$.
Consideration of the primary decomposition of $K$ 
shows that $K$ must be $m$-primary and $p(\Lambda)=\Lambda/K$ finite.
Hence $R$ is also finite.
\end{proof}

An object $N$ in an abelian category $\mathcal{C}$ is a {\it torsion} object
if the ring $End_{\mathcal{C}}(N)=\mathcal{C}(N,N)$ is finite.

\begin{theorem}
Let $Q=(A,\theta,\mathcal{E})$ in $\mathcal{Q}$.
Then the following are equivalent
\begin{enumerate}
\item $A$ is a finite $\Lambda$-module;

\item $Q$ is a torsion object;

\item $Q$ is a finite sum of objects whose endomorphism rings are local rings.
\end{enumerate}
The Krull-Schmidt Theorem holds in the full subcategory of $\mathcal{Q}$ 
determined by the torsion objects, and the endomorphism rings of such objects
are radically complete.
\end{theorem}

\begin{proof}
It is immediate that $End_\mathcal{Q}(Q)$ is commensurable with
$End_\Lambda(A)\times{End_\Lambda(A)}$.
Thus $End_\mathcal{Q}(Q)$ is finite if and only if $End_\Lambda(A)$ is finite.
It follows easily that $Q$ is a torsion object in
$\mathcal{Q}$ if and only if $A$ is finite. 
Thus $(1)\Leftrightarrow(2)$.

In a finite ring some power of each element is idempotent.
Thus a finite ring with no nontrivial idempotents is local.
Since idempotents in $\mathcal{Q}$ split it follows that
$(2)\Rightarrow(3)$.

If $End_\mathcal{Q}(Q)$ is local then it is finite, by the lemma, and so 
$Q$ is a torsion object. 
Thus $(3)\Rightarrow(2)$.

The final assertion is now clear, since the radical of an Artinian ring is
nilpotent.
\end{proof}

If $(A,\theta,\mathcal{E})$ is a torsion object in $\mathcal{Q}$ then
$\theta:A\cong{e^2A}$ and $e^2(A/2A)\cong{_2A}$,
and the self-dual torsion objects of $\mathcal{Q}$ are essentially 
equivalent to the self-dual objects of the category considered in \cite{Hi86}.
(However, the latter category is {\it not\/} the subcategory 
of torsion objects of $\mathcal{Q}$.)

In \cite{FH93, Hi86} and \cite{HK89} certain classes of knots were classified by
similar reductions to pairings over finite fields, and in \cite{FH93} and \cite{Hi86} 
these invariants were applied to questions related to double null concordance for such knots.
Kearton had earlier shown that a torsion-free simple $2q$-knot is doubly null concordant
if and only if its $F$-form was hyperbolic, in an appropriate sense \cite{Ke84}.
Thus we expect that 
{\it
a simple $2q$-knot $K$ with $q\geq4$ is doubly null concordant if and only if
$q_F(K)$ is hyperbolic.}
Moreover,  
{\it a simple $2q$-knot $K$ with $q\geq4$ is stably doubly null concordant if and only if the
Witt class of $q_F(K)$ is trivial.}
Establishing the latter result may require extending $q_F(-)$ to
an invariant of all even-dimensional knots.

\end{document}